\documentclass[a4paper,10pt]{amsart}
\usepackage[utf8x]{inputenc}
\usepackage{ucs}
\usepackage[english]{babel}
\usepackage{amsmath}
\usepackage{amsfonts}
\usepackage{amssymb}
\usepackage{amsthm}
\usepackage{graphicx}
\usepackage{mathtools}
\usepackage{graphicx}
\usepackage{dsfont}
\usepackage{accents}
\usepackage{color}
\usepackage{bbding}
\usepackage[babel]{csquotes}
\usepackage[all]{xy}
\usepackage{tikz}
\usetikzlibrary{matrix,arrows}

\usepackage{version}
\usepackage{hyperref}
\usepackage{enumerate}

\newcommand{\C}{\mathbb{C}}

\newcommand{\ba}{\begin{array}}
\newcommand{\be}{\begin{equation}} 
\newcommand{\ea}{\end{array}}
\newcommand{\ee}[1]{\label{#1}\end{equation}}

\renewcommand{\b}[1]{\mathbf{#1}}

\newcommand{\ds}{\displaystyle}
\newcommand{\mcs}{\mathcal{S}}

\newcommand{\mcm}{\mathcal{M}}

\newcommand{\mcu}{\mathcal{U}}
\newcommand{\mcv}{\mathcal{V}}

\newcommand{\mcw}{\mathcal{W}}
\newcommand{\mcy}{\mathcal{Y}}

\newcommand{\mbp}{\mathbb{P}}

\newcommand{\sdp}{S^{[d]}_p}

\newcommand{\del}{\partial}

\renewcommand{\phi}{\varphi}
\renewcommand{\epsilon}{\varepsilon}

\newcommand{\frs}[2]{\hbox{$\ds \frac{#1}{#2}$}}

\DeclareMathOperator{\diag}{diag}		

\DeclareMathOperator{\Spec}{Spec}
\newtheorem{theorem}{Theorem}[section]

\newtheorem{lemma}[theorem]{Lemma}
\newtheorem*{lemma*}{Lemma}
\newtheorem{corollary}[theorem]{Corollary}
\newtheorem{prop}[theorem]{Proposition}
\newtheorem{definition}[theorem]{Definition}
\newtheorem*{definition*}{Definition}
\newtheorem{example}{Example}[section]
\newtheorem*{example*}{Example}
\newtheorem{rem}[theorem]{Remark}
\newtheorem*{rem*}{Remark}

\numberwithin{equation}{section}

\author{Niccol\`o Lora Lamia Donin}
\date{\today}
\title{Transverse Hilbert Schemes and Completely Integrable Systems}
\address{Institut f\"ur Differentialgeometrie, Leibniz Universit\"at Hannover
 \endgraf  Welfengarten 1, 30167
- Hannover - Germany}
\email{loralamia at math.uni-hannover.de} 
\begin{document}
\begin{abstract}
In this paper we consider a special class of completely integrable systems that arise as transverse Hilbert schemes of $d$ points of a complex symplectic surface $S$ projecting onto $\C$ via a surjective map $p$ which is a submersion outside a discrete subset of $S$. We explicitly endow the transverse Hilbert scheme $S^{[d]}_{p}$ with a symplectic form and an endomorphism $A$ of its tangent space with $2$-dimensional eigenspaces and such that its characteristic polynomial is the square of its minimum polynomial and show it has the maximal number of commuting Hamiltonians. We then provide the inverse construction, starting from a $2d$-dimensional holomorphic integrable system $\mcw$ which has an endomorphism $A \colon T\mcw \rightarrow T\mcw$ satisfying the above properties and recover our initial surface $S$ with $\mcw \cong \sdp$. \end{abstract}

\maketitle

\section{Introduction}
As it is well-known (see \cite{ACdS}), a Hamiltonian system is called \emph{completely integrable} if and only if it possesses the maximal number of independent Poisson-commuting first integrals of motion. In the present article we shall work in the holomorphic category and associate to a complex surface endowed with a holomorphic symplectic form and a Hamiltonian function a holomorphic completely integrable system of complex dimension $2d$. Our aim is to describe such a construction and provide a full characterization of the completely integrable systems that arise from it.\\
We will proceed according to the following plan.\\
In section 2 we start from a complex surface $S$ holomorphically projecting onto $\C$ via 
a map $p$ and recall from \cite{AH} the definition of Hilbert scheme of $d$ points transverse to $p$, which we denote by $\sdp$. We also recall from \cite{AB83} and \cite{AH}, \cite{RB15} that $\sdp$ is an open subset of the full Hilbert scheme $S^{[d]}$ of length $d$ $0$-dimensional subschemes of $S$, which is smooth and is of complex dimension $2d$ hence so is $\sdp$. Next, we show how $p$ induces a natural endomorphism $A$ of the tangent bundle $T\sdp$ such that at every point of $\sdp$ its characteristic polynomial is the square of its minimal polynomial and its eigenspaces all have complex dimension $2$.\\
In section 3 (Proposition \ref{prop: endom inverse construction}) we show how the eigenvalues and eigenspaces of $A$ determine the geometry of $S$ and $\sdp$, meaning that, starting from a $2d$-$\dim_{\C}$ manifold $\mcw$ endowed with an endomorphism $A$ of $T\mcw$ with the above properties, $S$ is recovered as the leaf space of a foliation induced by $A$ and $\mcw$ is identified with $\sdp$.\\
Section 4 is finally devoted to the characterization of $\sdp$ as a holomorphic completely integrable system. Assuming that the surface $S$ carries a holomorphic symplectic $2$-form $\omega$ we recover Beauville's result (\cite{AB83}) and show that $\omega$ induces a symplectic structure on $\sdp$, for which we point out the existence of $d$ independent Poisson-commuting Hamiltonians, thus giving $\sdp$ the structure of a completely integrable system. The remainder of the section is then focused on reversing the construction: we start from a complex $2d$-dimensional completely integrable system, endowed with the extra feature of an endomorphism $A$ of its tangent bundle with the aforementioned properties and show how, under some compatibility assumptions between the symplectic form and $A$, $\mcw$ is recovered as the integrable system constructed as the transverse Hilbert scheme of $d$ points of a surface $S$ with a projection $p \colon S \rightarrow \C$.
\begin{rem}\rm
The motivation for this paper comes from the link between the theory of completely integrable systems and the geometry of manifolds of higher degree curves in twistor spaces. As explained in \cite{RB14}, given a holomorphic fibration $Z \xrightarrow{\pi} T\mbp^1 \xrightarrow{\pi_1} \mbp^1$ such that $Z$ has a real structure covering the antipodal map of $\mbp^1$ and $\dim_{\C}Z=3$, the space $\mcm_d$ of real ``degree $d$" curves in $Z$ satisfying appropriate conditions is a hyperkähler manifold of real dimension $4d$. When one complex structure $I_{\zeta}$ among the $S^2$ of possible complex structure is fixed, there is a finite unramified covering map
\begin{align}
\psi \colon \left(\mcm_d, I_{\zeta}\right) &\rightarrow \left(Z_{\zeta}^{[d]}\right)_{\pi}\\
C &\mapsto C_{\zeta},
\end{align}
where, keeping the notation of \cite{RB14}, $ \left(Z_{\zeta}^{[d]}\right)_{\pi}$ is the Hilbert scheme of $d$ points in $Z_{\zeta}\coloneqq \tilde{\pi}^{-1}(\zeta)$ transverse to the projection $\tilde{\pi}\vert_{Z_{\zeta}}=\left(\pi_1 \circ \pi\right)\vert_{Z_{\zeta}}$, and $C_\zeta=C \cap Z_{\zeta}$.\\
Following the construction of Sections 3 and 4, an endomorphism $A$ can be constructed on $T_{C_{\zeta}}\left(Z_{\zeta}^{[d]}\right)_{\pi}$ for every $C \in \mcm_d$, which has the mentioned properties on the eigenspaces and characteristic polynomial, fitting the above scenario to our description.\\
Examples of such manifolds are the ones obtained by the generalized Legendre transform such as the moduli space $\mcm_k$ of charge $k$ monopoles. As described by Atiyah and Hitchin in \cite{AH} we know that, once equipped with one chosen complex structure, $\mcm_k$ is diffeomorphic to the space $R_k$ of based rational maps of degree $k$, i.e.
\begin{align*}
R_k=\left\{ \frs{p(z)}{q(z)} \vert ~ \deg q=k, \deg p=k-1, q \text{ is monic and } p(z)\ne 0 \text{ if } q(z)=0 \right\}
\end{align*}
which, in turn, is equivalent to $S^{[k]}_{p}$ where $S=\C \times \C^*$ and $p$ is the projection onto the first factor.\\
\end{rem} 
\subsection*{Acknowledgements}
The author wishes to acknowledge Roger Bielawski for suggesting the problem as well as for his constant supervision and countless fruitful discussions. He also thanks Simon Salamon for his comments and suggestions during the stay in London and the DFG Graudiertenkolleg 1463 ``Analysis, Geometry and String Theory" for the financial support.

\
\section{An Endomorphism of the Tangent Space}
Let $S$ be a complex surface (for the moment being we do not assume it to be symplectic) with a surjective holomorphic projection $p \colon S \rightarrow \C$ which is a submersion outside a discrete subset $B \subset S$. We recall from \cite{RB15} the following definition of \emph{transverse Hilbert scheme of $S$ of $d$ points} with respect to the projection $p$.
\begin{definition}
The length $d$ Hilbert scheme of $S$ transverse to the projection $p$ is the subset $S^{[d]}_{p}$ of the full length $d$ Hilbert scheme $S^{[d]}$ of $S$ consisting of those $0$-dimensional subschemes $Z$ of length $d$ such that $p \colon Z \rightarrow p(Z)$ is an isomorphism onto the scheme-theoretic image.
\end{definition}

\begin{rem}\rm
It is clear from the definition that  $p$ induces a surjective holomorphic map $p^{[d]} \colon S^{[d]}_{p} \rightarrow \C^{[d]}\cong S^d(\C)$, where $S^d(\C)$ stands for the $d$-th symmetric power of $\C$.
\end{rem}

\begin{rem}\rm
A practical interpretation of the transversality condition is the following: taken $Z \in \sdp$, if $z \in p(Z) \subset \C$ is a point with multiplicity $k$, then it will correspond via $p$ to a point $s \in p^{-1}(z)$ also of multiplicity $k$. A scenario with points $s_1, \dots, s_j$ of multiplicity $k_1, \dots, k_j$, $\sum k_i=k$, all lying in $p^{-1}(z)$ is hence excluded.
\end{rem}

On the tangent space to $\sdp$ a natural endomorphism is defined as follows.
Let  $q(z)$ be the monic polynomial of degree $d$ defining the image $p(Z)$ of $Z$ via $p$ and observe that $H^0(Z,\C) \cong \C[z]/(q(z))$, where the generator $z$ stands as a preferred element.
Recall now that, for every $Z \in \sdp$, one has $T_Z\sdp \cong H^0(Z, TS\vert_Z)$ due to a well-known theorem of Kodaira (\cite{KOD62}).
 Then we set $A_Z$ to be the map 
\begin{align}
\begin{split}
H^0(Z, TS\vert_Z) &\longrightarrow H^0(Z,TS\vert_Z)  \\[6pt]
\sigma(z) & \mapsto f(z)\sigma(z),
\end{split}
\end{align}
where we take $f \colon X \rightarrow \C$ to be the function $z \in H^0(Z,\C)$.

\begin{rem}\rm
If $\sigma \in H^0(Z,TS\vert_Z)$ and $p\in S$ is a point of $Z$ with multiplicity one then $(A\sigma)(p)=z(p)\sigma(p)$. If, instead, $p$ has multiplicity $k>1$, we recall from  \cite[Proposition 2.4]{AHH90} that the section $\sigma$ is given as a power series in $(z-z(p))$ truncated at order $k$, that is 
\begin{align}\label{eq: power series sigma}
\sigma(z)=\sigma(z(p))+\sigma'(z(p))(z-z(p))+\dots+\frs{\sigma^{(k)}(z(p))}{k!}(z-z(p))^k.
\end{align}
Then $A\sigma$ will be give the truncated power series of $z\sigma(z)$, that is
\begin{align}\label{eq: power series A(sigma)}
(z\cdot\sigma)(z)&=z(p)\sigma(z(p))+(\sigma(z(p))+z(p)\sigma'(z(p)))(z-z(p))+\dots\\[6pt]
&+\frs{k\sigma^{(k-1)}(z(p))+z(p)\sigma^{(k)}(z(p))}{k!}(z-z(p))^k.
\end{align}
Comparing \eqref{eq: power series A(sigma)} and \eqref{eq: power series sigma} we deduce that the eigenspaces of $A$ are of dimension $2$. Also, the eigenvalues have even multiplicity each one equal to the dimension of the relative power expansion space. These two observations altogether yield, at each point of $\sdp$, the Jordan canonical form of $A$.
\end{rem}

\begin{example}[The space of rational maps] \label{ex: rat map and AH}\rm
The machinery we have introduced so far allows us to build such an endomorphism $A$ on the tangent space to the space of based rational maps of degree $d$. As an example we compute it for $d=2$.\\
Let us define the complex surface $S=\C \times \C^*$ projecting onto $\C$ via $p$ which we interpret as the moduli space of charge $1$ monopoles and let $S^{[2]}_{p}$ be its Hilbert scheme of points of length $2$ transverse to $p$. We identify (see \cite{AH}) $S^{[2]}_{p}$ with the space of all based rational map of degree $2$, defined by 
\begin{equation}
R_2=\left\{\frs{p(z)}{q(z)}=\frs{a_1z+a_0}{z^2-q_1z-q_0} \vert~ p(z) \text{ and } q(z) \text{ have no common roots } \right\}.
\end{equation}
Observe that a tangent vector to $R_2$ at a point $(p(z), q(z))$ is given as a couple of degree $1$ polynomials $(q'(z),p'(z))$ where we write $q'(z)=q'_1z+q'_0$ and $p'(z)=p'_1z+p'_0$. Applying the previous construction we get an endomorphism $A$ of the tangent bundle to $R_2$ which on the tangent space to $R_2$ at each point $(p(z),q(z))$ operates as multiplication by $z$ modulo $q(z)$. This means that
\begin{align*}
A_{(q(z),p(z))} \colon T_{(q(z),p(z))}R_2 & \longrightarrow  T_{(q(z),p(z))}R_2 \\[6pt]
(q'_1z+q_0, p'_1z+p_0) & \mapsto ((q_1q'_1+q'_0)z+q_0q'_1, (q_1p'_1+p'_0)z+q_0p'_1)
\end{align*}
is represented at $(p(z), q(z))$ by the block-diagonal matrix 
\begin{equation}\label{mult by z mod q}
A_{(q(z),p(z))} \coloneqq
\begin{pmatrix}
q_1 & 1 & 0 & 0 \\[6pt]
q_0	& 0 & 0 & 0 \\[6pt]
0 & 0 & q_1 & 1 \\[6pt]
0 & 0 & q_0 & 0 
\end{pmatrix}
\end{equation}
where each block is the so-called companion matrix of the polynomial $q(\lambda)$.\\
Let us now focus on the open dense subset of $R_2$ consisting of all based degree $2$ rational maps with simple poles. If a map $p(z)/q(z)$ has distinct poles, i.e. the roots of $q$ are distinct, then it can be identified with the point $X=\left((\beta_1, p(\beta_1), (\beta_2, p(\beta_2))\right) \in S^{[2]}_{p}$, where the $\beta_i$'s are the roots of $q$ and $p(z)$ is recovered by Lagrange interpolation as the unique linear polynomial taking the values $p(\beta_i)$ at $\beta_i$. The projection $p \colon S \longrightarrow \C$ induces on every $X \in S^{[2]}_{p}$ a function $f \colon X \longrightarrow \C$ taking $(\beta_i, p(\beta_i))$ into $\beta_i \in \C$. Using the fact that $T_XS^{[2]}_{p} \cong H^0(X,TS\vert_X)$, the function $f$ induces an endomorphism $H^0(X,TS\vert_X) \longrightarrow H^0(X,TS\vert_X)$ given by $\sigma(x) \longrightarrow f(x)\sigma(x)$ for $x \in X$. In the tangent frame provided by these coordinates, $A$ at $(\beta_i, p(\beta_i))$ is represented by the diagonal matrix $\diag (\beta_1, \beta_1, \beta_2, \beta_2)$. Since on this open subset $q_0=-\beta_1\beta_2, q_1=\beta_1+\beta_2$, a computation shows that this diagonal matrix actually is the Jordan canonical form of \eqref{mult by z mod q}. \\
We observe also that, when $q_1=2\beta$ and $q_0=-\beta^2$ i.e. the rational map has a double pole at $z=\beta$, then the Jordan form of $A$ is 
\begin{equation}
\begin{pmatrix}
\beta & 1 & 0 & 0 \\[6pt]
0& \beta & 0 & 0 \\[6pt]
0 & 0 & \beta & 1 \\[6pt]
0 & 0 & 0 & \beta 
\end{pmatrix}
\end{equation}
\end{example}

\begin{example}\rm
Let us consider the double cover of the Atiyah-Hitchin manifold. As described in \cite{AH} this is a surface $S \subset \C^3$ defined by $S=\left\{(z,x,y) \vert~ x^2-zy^2=1 \right\}.$ We can therefore consider the Hilbert scheme of $d$ points of $S$ tranverse to the projection $p$ onto the first coordinate. We recall from \cite{RB15} that it can be described as the set of triple of polynomials $x(z), y(z), q(z)$ such that $x(z)$ and $y(z)$ have degree $d-1$, $q(z)$ is monic of degree $d$ and the equation $x^2(z)-zy^2(z)=1$ modulo $q(z)$ is verified.\\
An alternative description (also explained in \cite{RB15}), which we will use here, is obtained by considering the quadratic extension $z=u^2$. In this case the equation $x^2-zy^2=1$ is rewritten as $(x+uy)(x-uy)=1$ and we observe that $p(u)=x(u^2)+uy(u^2)$ is a polynomial of degree $2d-1$ in $u$ while $q(u^2)$ is a polynomial of degree $2d$ in $u$ which has no odd terms. The Hilbert scheme $\sdp$ is then described as the set of all couples of polynomials $(p(u),q(u^2))$ such that $p(u)p(-u)=1$ modulo  $q(u^2)$. Similarly to the previous example, a tangent element in $T_{(p(u),q(u^2))}\sdp$ is given by a couple of polynomials of the form
\begin{align}
p'(u)&=p'_0+p'_1u+\dots+p'_{2d-1}u^{2d-1}\\
q'(u^2)&=q'_0+q'_1u^2+\dots+q'_{d-1}u^{2d-1}
\end{align}
such that 
\begin{align}\label{HS of AH cover v2: tangent condition}
p'(u)p(-u)+p(u)p'(-u)=0 \text{ modulo } q(u^2).
\end{align}
We finally produce the endomorphism $A \colon T_{(p(u),q(u^2))}\sdp \rightarrow T_{(p(u),q(u^2))}\sdp$ at every $(p(u),q(u^2)) \in \sdp$ as multiplication by $u^2$ modulo $q(u^2)$, after observing that it preserves the space of solutions to \eqref{HS of AH cover v2: tangent condition}.
\end{example}

We now define the manifold $\tilde{S}^{[d]}_p$as the set of all $(z, Z) \in \C \times \sdp$ such that $z  \text{ is eigenvalue of }  A_Z$ and observe that it comes with a double projection 
\begin{equation}
\xymatrix{
&\tilde{S}^{[d]}_p \ar[d]_{\rho} \ar[dr]^{\pi}\\
&\C  &\sdp
}
\end{equation}
where $\pi$ is a branched $d:1$ covering of $\sdp$. Also, for every $X \in \sdp$, one can lift $A_Z$ to an endomorphism $\pi^*T_Z\sdp \longrightarrow \pi^*T_Z\sdp.$ Hence we draw the following diagram
\begin{equation}\label{eq: two sequences}
\footnotesize
\begin{tikzpicture}
\matrix (m) [matrix of math nodes, row sep=3em,
column sep=3em, text height=1.5ex, text depth=0.25ex]
{&		&			&0		&\\
&		&			&T^V_pS\vert_Z 		&\\
0 &  \pi^*T_Z\sdp  &  \pi^*T_Z\sdp &TS\vert_Z & 0 \\
 &		& 			&	T\C	& \\
&&  &0&\\};
\path[->] (m-1-4) edge (m-2-4);
\path[->] (m-2-4) edge (m-3-4);
\path[->](m-3-1) edge (m-3-2);
\path[->] (m-3-2) edge node[auto]{$z-A_Z$} (m-3-3);
\path[->] (m-3-3) edge (m-3-4);
\path[->](m-3-4) edge (m-3-5);
\path[->](m-3-4) edge (m-4-4);
\path[->](m-4-4) edge (m-5-4);
\path[->] (m-3-3) edge[dotted] (m-4-4);
\end{tikzpicture}
\end{equation}
Let $\beta$ be the function defined by the dotted arrow. We see that $Im(z-A_Z)$ lies in the kernel of $\beta$. Also, one has that elements of $\pi^*T_Z\sdp$ correspond to deformations of $\sdp$ at $Z$ and elements in $T^V_{\rho}\tilde{S}^{[d]}_p$ correspond to deformations fixing the eigenvalue. From this we get that $\ker \beta=T^V_{\rho} \tilde{S}^{[d]}_p$ and $Im(z-A) \subset T^V_{\rho}\tilde{S}^{[d]}_p$. 

\begin{rem}\rm
The holomorphic distribution $D\coloneqq Im(z-A)$ defined on $\tilde{S}^{[d]}_p$ is clearly involutive on the dense subset of $\tilde{B} \subset \tilde{S}^{[d]}_p$ consisting of all couples $(z, Z) \in \C\times \sdp$ with $Z$ a length $d$ $0$-dimensional subscheme of $S$ consisting of points that are all distinct, hence it is involutive on the whole $\tilde{S}^{[d]}_p$.
\end{rem}
\noindent
Construct now the double fibration 
\begin{equation}
\xymatrix{
&\mcy \ar[d] \ar[dr]\\
&S  &\sdp
}
\end{equation}
the manifold $\mcy$ being defined as $\mcy=\left\{(s,Z) \in S \times \sdp \vert~ s\in Z \right\}$. So far we notice that 
\begin{align*}
\tilde{S}^{[d]}_p &= \left\{ (z,Z) \in \C \times \sdp \vert~z \text{ is eigenvalue of } A_Z \right\}\\[6pt]
&=\left\{ (z,Z) \in \C \times \sdp \vert~z=p(x) \text{ for some } x \in Z\right\}\\[6pt]
&\cong \left\{(x,Z) \in S \times \sdp \vert ~ x \in Z \right\}.
\end{align*}
Hence we recover our initial surface as the space of leaves $S=\mcy/D \cong \tilde{S}^{[d]}_p/D$. \\
This suggests us the following \emph{inverse construction}.

\section{The Inverse Construction}
Let us start with a complex manifold $\mcw$ of complex dimension $2d$ endowed with an endomorphism $A: T\mcw \rightarrow T\mcw$ of its holomorphic tangent bundle $T\mcw$ with eigenvalues of even multiplicity and such that its characteristic polynomial is the square of its minimal polynomial. Set now $X \coloneqq \C^{[d]}$ Hilbert scheme of $d$ points of $\C$ and define a map $\mu \colon \mcw \longrightarrow X$ which assigns to each point $w \in \mcw$ the minimal polynomial of $A$ at $w$ which we denote $q_w(\lambda)$. Assume now $\mu$ to be a surjective submersion and define a vector field $\b{V} \in \mathfrak{X}(\mcw)$ to be projectable for $\mu$ if, for every $x \in X$, $d\mu_w(\b{V}_w)$ does not depend of the choice of $w$ in $\mu^{-1}(x)$. If we suppose that $A$ preserves the vertical vectors and the projectable vector fields for the projection $\mu$, then it descends to a map $\bar{A} \colon TX \longrightarrow TX$ which makes the following diagram commute
\begin{equation} \label{diag: A Abar}
\xymatrix{
&T\mcw \ar[d]_{d\mu} \ar[r]^{A}  &T\mcw \ar[d]^{d\mu}\\
&TX \ar[r]_{\bar{A}} &TX.
}
\end{equation}
\begin{definition}
If an endomorphism $A$ that satisfies the above conditions is such that none of its generalized eigenspaces is fully contained in $\ker(d\mu)$, then we will call it \emph{compatible} with the projection $\mu$ defined by its minimal polynomial.
\end{definition}
For $q(\lambda) \in X$ let us identify $T_{q(\lambda)}X \cong \C[\lambda]/ (q(\lambda))$ and assume that $A$ is compatible with $\mu$. Then $\bar{A}$ is naturally given by multiplication by $\lambda$ modulo $q(\lambda)$. 
Set also $\mcw_z=\left\{ w \in \mcw \vert~ z \in \Spec A_w\right\}$, i.e. $\mcw= \pi^{-1}(X_z)$ where $X_z$ is the set of all monic polynomials of degree $d$ for which $z$ is a root. With these definitions we see for a tangent vector $\b{V}$ that $\b{V} \in T\mcw_z \Longleftrightarrow d\mu(\b{V}) \in T X_z$, where the tangent space to $X_z$ at $q(\lambda)$ can be described as 
\begin{equation}
T_qX_z=\left\{p(\lambda) \vert ~ \deg p(\lambda)=d-1 \text{ and } p(z)=0 \right\}.
\end{equation}

Take now a polynomial $q'(\lambda) \in T_{q(\lambda)}X$ and $z \in C$: the definition of $\bar{A}$ implies that $(z\mathds{1}-\bar{A})(q'(\lambda))$ is a polynomial of $T_{q(\lambda)}X$ that vanishes at $z$ that is, by the commutativity of the diagram, $Im(z\mathds{1}-A) \subset T\mcw_z$.\\
Define now $\tilde{\mcw}=\left\{(z,w) \in \C \times \mcw \vert~ z \text{ is an eigenvalue of }A_w \right\}$, which is a $d \colon 1$ covering of $\mcw$, with  two projections 
\begin{equation}\label{diag: pi rho}
\xymatrix{
&\tilde{\mcw} \ar[d]_{\rho} \ar[dr]^{\pi}\\
&\C  &\mcw.
}
\end{equation}
Then $A$ can be lifted to an endomorphism of $T(\C \times \mcw)$, which we will still denote by $A$, preserving the vertical subbundles of $\rho$ and $\pi$. The previous observations imply that, at every point $(z,w)$, $A$ acts on the vertical subbundle of $\pi$ as multiplication by $z$ and that it descends to $T\tilde{\mcw}$.
Assuming now that the distribution $Im\left(z\mathds{1}-A\right)$ is integrable, we see that it defines a subdistribution of the integral distribution $\ker d\rho$.
We can therefore recover our initial surface $S$ as the leaf space
\begin{equation}
S \coloneqq \frs{\tilde{\mcw}}{Im\left(z\mathds{1}-A\right)}
\end{equation}
The surface $S$ comes with a natural projection $p \colon S \longrightarrow \C$ defined as $p([(z,w)])=z$, which makes the following diagram commute
\begin{equation}
\xymatrix{
&S \ar[d]_p & \tilde{\mcw} \ar[l]_{proj} \ar[dl]^{\rho}\\
&\C
}
\end{equation}
It is now sufficient to define $\mcu \subset S^{[d]}$ as $\mcu=\left\{proj(\pi^{-1}(w)) \vert ~ w \in \mcw \right\}$ and $Z=Z(w)=proj(\pi^{-1}(w)) \in \mcu$ in order to apply the previously exposed construction for getting $A_X \colon T_X\mcu \longrightarrow T_X\mcu$ i.e. once more our endomorphism $A_w \colon T_w\mcw \longrightarrow T_w\mcw$ for every point $w$ of $\mcw$.\\
Hence, keeping the conventions that we have introduced so far, we have proven the following.

\begin{prop} \label{prop: endom inverse construction}
Let $\mcw^{2d}$ be a complex manifold of complex dimension $2d$ with the following properties.
\begin{enumerate}[(i)]
\item $\mcw$ comes with an endomorphism $A \colon T\mcw \rightarrow T\mcw$ such that at every point the eigenspaces have complex dimension $2$ and the characteristic polynomial is the square of the minimal polynomial 
\item Assume that $A$ is compatible with the induced projection $\mu \colon \mcs \rightarrow X \coloneqq \C^{[d]}$, so that Diagram \eqref{diag: A Abar} is defined 


\item The distribution $D \coloneqq Im(z-A)$ is integrable on the incidence manifold $\tilde{\mcw}=\left\{(z,w) \in \C \times \mcw \vert~ z \text{ is an eigenvalue of }A_w \right\}$.
\end{enumerate}
Then $\rho \colon S\coloneqq \tilde{\mcw}/ D \rightarrow \C$ is a surface projecting on $\C$ for which $\mcw$ is the length $d$ Hilbert scheme of points transverse to the projection.
\end{prop}

\section{A Symplectic Form}
This section will be devoted to revising, in both directions, the previously exposed construction when we assume our surface $S$ to carry a symplectic form $\omega$ on its tangent bundle. From now on we shall also assume that the projection $p \colon S \rightarrow \C$ is submersion outside at most a discrete subset $B \subset S$.\\
As we have already pointed out in the Introduction, the results of Sections $3$ and $4$ show that the transverse Hilbert scheme of points $\sdp$ gets the structure of a holomorphic completely integrable system. The importance of $A$ in distinguished whether a given holomorphic integrable system arises as the Hilbert scheme of points of a holomorphic symplectic surface is well motivated by the following example.

\begin{example}[Motivational Example]\rm
Let us consider the complex $2d$-dimensional manifold $\C^{2d}$, with coordinates $(z_i,t_i), i=1,\dots,d$ and endowed with the standard symplectic form $\Omega_0=\sum dz_i \wedge dt_i$, endowed with he projection $p_0\colon \C^{2d} \rightarrow \C^d$ given by $p_0(z_i,t_i)=(z_i)$. Observe that the coordinates $z_i$ are $d$ commuting Hamiltonian functions. Now, the projection $p_0$ induces an endomorphism $A\colon T\C^{2d} \rightarrow T\C^{2d}$ given, at every point of $\C^{2d}$, by the diagonal matrix $\diag(z_1,z_1, \dots, z_d, z_d)$. The eigenspaces of this endomorphism, however, show a jump in dimension whenever two eigenvalues happen to coincide. As a result of this, although $(\C^{2d},\Omega_0, z_1, \dots, z_d)$ is a holomorphic completely integrable system, it does not arise via a transverse Hilbert scheme construction as $A$ does not meet the necessary requirements.\\
We now consider a different $\C^{2d}$ from the one above, with a different projection. Namely we take the space $X_d$ of all couples of polynomials $\left( Q(\lambda), P(\lambda)\right)$ such that $Q$ is monic of degree $d$ and $T$ has degree $d-1$. If we write $Q=\lambda^d-\sum Q_j \lambda^j$ and $t=\sum T_i \lambda^i$ then $(Q_i,T_i)$ are global coordinates and $X\cong \C^{2d}$. Now, on the open dense subset $\mcv$ of $X$ consisting of couples $(Q(\lambda), T(\lambda)$ such that $Q$ has all distinct roots, we also have coordinates $(\beta_i, T(\beta_i))$ where $\beta_i$ are the roots of $Q$ and $T(\beta_i)$ the values of $T$ on those roots. In these latter coordinates the form $\sum d\beta_i \wedge T(\beta_i)$ is defined and it can uniquely extended to a holomorphic symplectic $2$-form $\Omega$ on the whole of $X_d$. The functions $Q_0, \dots, Q_{d-1}$ will then be $d$ commuting Hamiltonians with respect to $\Omega$: in fact, on the open dense subset $\mcv$ they are just the elementary symmetric polynomials in the roots $\beta_i$, hence they commute with each other on $\mcv$, so on all $X_d$. We set then the projection $p \colon X_d \rightarrow \C^d$, $p(Q_i, T_i)=(Q_i)$ and observe that $(X_d, \Omega, Q_0, \dots, Q_{d-1})$ is a holomorphic completely integrable system. The projection $p$ defines here an endomorphism $A$ of $TX_d$ which is represented by 
\begin{align}
\begin{pmatrix}
C_Q & 0 \\[6pt]
0 & C_Q
\end{pmatrix},
\end{align}
where $C_Q$ is the so-called companion matrix of the polynomial $Q(\lambda)$ and meets our requirements. Therefore, thanks to our results of sections 3 and 4, we can recover the holomorphic completely integrable system $(X_d, \Omega, Q_0, \dots, Q_{d-1})$ as the Hilbert scheme of $d$ points of the surface $\C \times \C$ transverse to the projection onto the first coordinate.

\end{example}

In \cite[Proposition 5]{AB83} Beauville proves that the full Hilbert scheme $S^{[d]}$ of a complex symplectic surface $(S,\omega)$ has a symplectic form induced by $\omega$. In the following Lemma we will explicitly recover his result on the transverse Hilbert scheme of $d$ points $\sdp$, which we know to be an open subset of the full Hilbert scheme. We remark that the existence of a symplectic form on the Hilbert scheme of $d$ points in $\C \times \C^*$ transverse for the projection $p \colon \C \times \C^* \rightarrow \C$ onto the first coordinate was pointed out by Atiyah-Hitchin in \cite[Chapter 2]{AH}, where an explicit formula is only given on the subset $\mcv \subset (\C \times \C^*)^{[d]}_p$ of $d$-tuples consisting of all distinct points.

\begin{lemma}\label{lemma: induced form on HS}
Let $p \colon S \rightarrow \C$ be a complex surface projecting onto $\C$ and assume that $p$ is a submersion outside a discrete set $B \subset S$. Assume also that $S$ possesses a holomorphic symplectic form $\omega$ Then $\omega$ induces a symplectic form $\Omega$ on the Hilbert scheme $\sdp$ of $d$ points in $S$ transverse to $p$.

\begin{proof}
We start by proving the Lemma in the case $d=2$.\\
Fix $d=2$ and define $M=\left(S\setminus B\right)^{[2]}_p$. Let $\mcv \subset S^{[2]}_p$ be the set of all $0$-dimensional subschemes of $S$ consisting of two distinct points. For every $Z \in M \cap \mcv$, i.e. consisting of two different points $p_1, p_2$ of $S$, the isomorphism $T_ZM \cong H^0(Z,TS\vert_Z)$ easily yields the symplectic form on $M \cap \mcv$:the map $\psi \colon S \times S \rightarrow S^{[2]}_p$ has no ramification on $\mcv$ hence is a $2$ to $1$ covering. By breaking the $\mathfrak{S}_2$-symmetry and choosing a sheet of $\psi$, that is ordering the couple $p_1, p_2$, one splits $H^0(Z,TS\vert_Z) \cong T_{p_1}S \oplus T_{p_2}S$  and defines $\Omega_Z=\omega_{p_1} \oplus \omega_{p_2}$. Since the local coordinates $(z,t)$ on $S \setminus B$ induce local coordinates $(z_1, t_1, z_2, t_2)$ around each $Z $ in $M \cap \mcv$ simply by evaluating on the points $p_1$ and $p_2$ of $Z \in M\cap \mcv$, we can locally write $\Omega=dz_1 \wedge dt_1 +dz_2 \wedge dt_2$ on $M \cap \mcv$. \\
We now have to extend the form $\Omega$ to those elements $W \in M$ which consist of one point $s \in S \setminus B$ taken with double multiplicity. In order to do so we adapt a construction by Bielawski, \cite{RB15}, in the following way.\\
Let $W \in M$ be as above and observe that since $p$ is a submersion on $M$, we can choose local coordinates $(z,t)$ around $s$ such that the first one is the base coordinate of $\C$. Moreover, we can choose them in such a way that $\omega= dz \wedge dt$: if this was not the case, i.e. $\omega= \omega(z,t) dz \wedge dt$, we could define a Darboux coordinate chart $(z, u)$ around $s$ simply by choosing a new holomorphic fibre coordinate $u$ such that $\del u / \del t = \omega(z,t)$. Of course such a $u$ can always be found as it amounts to finding a primitive of a holomorphic function on a simply connected domain. We then describe the open set $\mcu^{[2]}_s \subset M$ as the set of couples of polynomials $(q(z), t(z))$ such that $q$ is monic of degree $2$ and $t$ is linear, that is $q(z)=z^2-Q_1z-Q_0,~ t(z)=T_0+T_1z$. On $\mcu^{[2]}_s \cap \mcv$, i.e. where $q(z)$ has distinct roots $z_1$ and $z_2$ the polynomial $t(z)$ can be recovered by Lagrange interpolation from the values $t_1=t(z_1)$ and $t_2=t(z_2)$: this gives an equivalence between the two sets of coordinates $(z_1, t_1, z_2, t_2)$ and $(Q_0, Q_1, T_0,T_1)$. At this point we observe that the form $\Omega$ can be rewritten in the coordinates $(Q_i,T_i)$ as  $\Omega=Q_1dT_1\wedge dQ_1 + dT_1 \wedge dQ_0 + dT_0 \wedge dQ_1$, which is well defined, closed and non degenerate on the whole $\mcu^{[2]}_p$. Since, as we will prove in the next Lemma, this construction is independent of the choice of local coordinates $\omega$ induces a holomorphic symplectic form $\Omega$ on $M$. As $B$ is discrete in $S$ then $S^{[2]}_p \setminus M$ has codimension at least $2$ in $S^{[2]}_p$ therefore $\Omega$ extends to the whole $S^{[2]}_p$ by Hartog's Theorem.\\
In the $d>2$ case one proceeds exactly as above to get a form $\Omega$ defined on the set of all $Z \in \sdp$ consisting either of $d$ distinct points or of $(d-2)$ distinct points and one point which is taken with double multiplicity. Since the remaining subset has codimension greater than $2$ in $\sdp$ again the form $\Omega$ extends to the whole $\sdp$ by Hartog's Theorem.
\end{proof}
\end{lemma}

We now show that this construction does not depend of the choice of coordinates.
\begin{lemma}
The construction of Lemma \ref{lemma: induced form on HS} is independent of the choice of coordinates.
\begin{proof}
Again we start from the $d=2$ case, keeping the notation of the previous lemma.
Let $(z,t)$ and $(z,w)$ be two sets of Darboux coordinates adapted to the projection $p$ on $(S, \omega)$. Denote by $\phi$ the change of coordinates between $(z,t)$ and $(z,w)$, so that $\omega'=(\phi^{-1})^*\omega$ is defined. The same observation as in Lemma \ref{lemma: induced form on HS} yields a symplectic form $\Omega'$ on $M \cap V$ and a change of coordinates $\Phi$ such that $\Omega'=\Phi^* \Omega$.\\
Let now $E \in M$ be an element of $S^{[2]}_p$ consisting of a point $s \in S$ taken with double multiplicity and let $\mcu'$ be a coordinate neighbourhood of $s$ for the coordinates $(z,w)$. The description of $(\mcu')^{[2]}_p$ is then 
\begin{align*}
(\mcu')^{[2]}_{p}=\{(q(z),w(z)) \vert~ &\deg q(z)=2,\deg w(z)=1,\\[6pt]
&q \text { monic } \}
\end{align*}
where 
\begin{align*}
q(z)&=z^2-Q_1z-Q_0 \\[6pt]
w(z)&=W_1z+W_0.
\end{align*}
Observe that $(Q_i, W_i)$ are local coordinates on $(\mcu')^{[2]}_p$. By abuse of notation, we keep denoting by $\Phi$ the change of coordinates between $(Q_i, T_i)$ of Lemma \ref{lemma: induced form on HS} and $(Q_i, W_i)$ on the intersection $(\mcu)^{[2]}_p \cap (\mcu')^{[2]}_p$. Then $(\Phi^{-1})^*\Omega$ is defined on all $(\mcu)^{[2]}_p \cap (\mcu')^{[2]}_p$ and coincides with $\Omega'$ on $(\mcu')^{[2]}_p \cap (\mcu)^{[2]}_p \cap \mcv$, therefore being its unique holomorphic extension. We conclude by extending $\Omega'=(\Phi^{-1})^*\Omega$ to the whole $S^{[2]}_p$ via Hartog's Theorem.\\
The generalization to greater $d$ is again achieved by applying the $d=2$ construction to the subset of  all elements in $\sdp$ consisting of $d$ distinct points or $(d-2)$ distinct points and one double point and then by extension via Hartog's Theorem.
\end{proof}
\end{lemma}

\begin{corollary}\label{cor: Compatibility A Omega}
The endomorphism $A$ and the symplectic form $\Omega$  satisfy the condition $\Omega(A\cdot, \cdot)=\Omega(\cdot, A\cdot)$.
\begin{proof}
It suffices to show the claim on the open dense subset $\mcv\subset\sdp$ of elements consisting of all distinct points. But there we can use coordinates that are both Darboux for $\Omega$ and diagonalizing $A$, so the assertion is trivially verified.
\end{proof}
\end{corollary}

\begin{corollary}
The transverse Hilbert scheme of points $\sdp$ is a holomorphic completely integrable system
\begin{proof}
This is an immediate consequence of Corollary \ref{cor: Compatibility A Omega}: the coefficients  $Q_i$ of the minimal polynomial of $A$ are $d$ Poisson-commuting functions for the Poisson structure associated to $\Omega$ on the dense subset $\mcv$, hence on all $\sdp$.
\end{proof}
\end{corollary}

\begin{example}\rm
A basic example summarizing what we have done is given by taking $S=\C \times \C$, $p \colon \C \times \C \rightarrow \C$ defined by $(x,y) \mapsto z=xy$ and $\omega=dx\wedge dy$. Of course $p$ is a submersion away from the origin and $\{(0,0)\}$ is a codimension $2$ subset of $S$. Hence on $(\C \times \C) \setminus_{\{(0,0)\}}=\{x≠0\} \cup \{y≠0\}=\mcu_1 \cup \mcu_2 $ we proceed exactly as in Example \ref{ex: rat map and AH} and apply our construction taking coordinate $(z, \chi_1)$ on $\mcu_1$ and $(z, \chi_2)$ on $\mcu_2$ where $\chi_1=-\log(x)$ and $\chi_2=\log(y)$. Observe that on $\mcu_1$ we write $\omega=dz\wedge d\chi_1$ and $\omega=dz \wedge d\chi_2$ on $\mcu_2$ and that they agree on the overlap $\mcu_1 \cap \mcu_2$. Each patch can be described as the set $\{(q(z), p(z)) \vert ~ q \text{ is monic of $\deg d$},~p \text{ has $\deg d-1$ and } p(0)≠0 \text{ if } q(0)=0 \}$. The construction now yields the symplectic form $\Omega$ on $\left((\C \times \C) \setminus B\right)^{[2]}_{p}=\{E \in (\C \times \C)^{[2]}_{p} \vert~ (0,0) \notin E \}.$ Since the complementary set to $\left((\C \times \C) \setminus B\right)^{[2]}_{p}$ has codimension $2$, we get $\Omega$ on the whole $(\C \times \C)^{[2]}_{p}$ applying Hartog's Theorem.
\end{example}

In the following proposition we work out the inverse construction in order to recover the holomorphic $2$-form initially given on the surface $S$ starting from the induced completely integrable system.
\begin{prop}
Let $\mcw$ be a complex manifold of complex dimension $2d$ endowed with an holomorphic endomorphism of the tangent space $T\mcw$ as in Proposition \ref{prop: endom inverse construction}. Assume also that $\mcw$ possesses a symplectic form $\Omega$ such that
\begin{itemize}
\item $\Omega(A \cdot, \cdot)=\Omega(\cdot,A\cdot)$
\item Each fiber $\mu^{-1}(x)$, for $x \in X=\C^{[d]}$ is Lagrangian, i.e. the vertical subbundle $\ker (d\mu)$ is maximal $\Omega$-isotropic.
\end{itemize}
Then $S=\tilde{\mcw}/D$ has a symplectic form induced by $\omega$
\end{prop}
\begin{proof}
We will obtain a symplectic form on $S=\tilde{\mcw}/D$ by defining on $\tilde{\mcw}$ a $2$-form $\tau$ of the form $\tau=\rho^*dz \wedge \alpha_z$ in the notation of Diagram \ref{diag: pi rho}, with $\alpha_z$ a $1$-form on $\tilde{\mcw}_z=\rho^{-1}(z)=\{ (w,z) \in \tilde{\mcw} \vert~ z \text{ is eigenvalue of } A_w\}$, such that $\tau(X,Y)=0$ for every $X \in T\tilde{\mcw}$ and $Y \in D$.\\
By the commutativity of Diagram \ref{diag: pi rho} and the surjectivity of $d\mu$, $\ker(z\mathds{1}-A)$ maps surjectively onto $\ker(z\mathds{1}-\bar{A})$. At this point we define $\mcw_z=\pi(\tilde{\mcw}_z)$ and we observe that the restriction $\pi \vert_{\tilde{\mcw}_z} \colon \tilde{\mcw}_z \rightarrow \mcw_z$ is obviously a diffeomorphism for every $z \in \C$m therefore $\mcw_z$ comes comes with a manifold structure. We can then choose a vector field $\b{v}_z$ on $X_z=\mu(\mcw_z)$ such that $\b{v}_z \in \ker(z\mathds{1}-\bar{A})$ at every point of $X_z$ and lift it to a vector $\b{V}_z$ tangent to $\mcw_z$ with $\b{V}_z \in \ker(z\mathds{1}-A)$. This lift is not uniquely determined: if $\b{V}'_z \in \ker(z\mathds{1}-A)$ is a second such lift, then $\b{V}_z-\b{V}'_z \in \ker(z\mathds{1}-A) \cap \ker{d\mu}$.\\
Define now the $1$-form $\alpha_z= \iota_{\b{V}_z} \Omega$ on $T\mcw_z$. This definition does not depend of the choice of $\b{V}_z$. In fact, at every $w \in \mcw_z$, one can split the tangent space to $\mcw_z$ as
\begin{equation}
T_w\mcw_z \cong Im(z\mathds{1}-A) \oplus \left\langle L \right\rangle 
\end{equation}
where $L \in \ker(z\mathds{1}-A) \cap \ker(d\mu)$. Now, since $\Omega(A \cdot, \cdot)=\Omega(\cdot, A\cdot)$ we have that $\ker(z\mathds{1}-A)$ and $Im(z\mathds{1}-A)$ are $\Omega$-orthogonal. This implies $\iota_{\b{V}_z-\b{V}'_z}\Omega(X)=0$ for every $X \in Im(z\mathds{1}-A$. Also, since $\ker(d\mu)$ is Lagrangian by assumption, we have $\iota_{\b{V}_z-\b{V}'_z}\Omega(L)=0$. Hence $\iota_{\b{V}_z-\b{V}'_z}\Omega=\alpha_z-\alpha'_z=0$ on all $\mcw_z$, meaning $\alpha_z$ is well defined.\\
Since $\pi\vert_{\tilde{\mcw}_z}$ is a diffeomorphism for every $z \in \C$ , $d\pi$ is an isomorphism and we can therefore pull $\alpha_z$ back to $\tilde{\mcw}_z$ via $\pi$  and define $\tau=\rho^* dz \wedge \pi^* \alpha_z$. As $D=Im(z\mathds{1}-A)$ satisfies $\tau(\cdot,D)=0$, the form $\tau$ descends to a form $\bar{\tau}$ on $S=\tilde{\mcw}/D$. We now prove that $\bar{\tau}$ is symplectic. First of all, $d\bar{\tau}=0$ as $\bar{\tau}$ is a $2$-form on a $2$-dimensional space. In order to prove its non-degeneracy we proceed as follows. First we observe that at every point $[(z,w)]$ of $S$ we have 
\begin{equation}
T_{[(z,w)]}S \cong \left\langle Y \right\rangle \oplus T_w \tilde{\mcw}_z/D_{(z,w)}
\end{equation}
where $Y$ is a vector in $T_w\tilde{\mcw}$ such that $d\rho (Y)=\del/\del z$. Take now $W \in T_w \tilde{\mcw}_z$ such that $[W] ≠0$ in $T_w \tilde{\mcw}_z/D$ and compute 
\begin{align}
(\rho^* dz \wedge \pi^*\alpha_z)_{[(z,w)]}(Y,W)=\pi^*\alpha_z(W)=\Omega((\b{V}_z)_w, d\pi_w(W)) ≠0
\end{align}
otherwise we would have $(\b{V}_z)_w \in (T_w\mcw_z)^{\Omega}$, where we denote with the superscript $\Omega$ the symplectic orthogonal complement. Now one observes that because both $Im(z\mathds{1}-A)$ and $\ker(d\mu)$ are contained in $T_w\mcw_z$ then $(T_w\mcw_z)^{\Omega} \subseteq\ker(z\mathds{1}-A) \cap \ker(d\mu)^{\Omega}=\ker(z\mathds{1}-A) \cap \ker(d\mu)$ as $\ker(d\mu)$ is Lagrangian. By counting dimensions we actually have $(T_w\mcw_z)^{\Omega} =\ker(z\mathds{1}-A) \cap \ker(d\mu)$. But this would imply $\b{V}_z \in \ker(d\mu)$ at $w$, which is in contrast with the fact that $\b{V}_z$ was constructed as a lift of a vector field $\b{v}_z$.\\
As a last step we prove that when $(\mcw, \Omega)$ is constructed as the transverse Hilbert scheme of a symplectic surface $(S,\omega)$ projecting onto $\C$ via $p$ with the symplectic form $\Omega$ induced by $\omega$ then, once we recover $S$ as $\tilde{\mcw}/D$ we also get back the original symplectic form $\omega$.\\
Let us write $\omega=dz \wedge \phi_z$, where $\phi_z$ is a $1$-form defined on the fibre $p^{-1}(z)$. Then on the usual open dense subset $\mcv \subset \mcw$ of all $d$-tuples of distinct points we have $\Omega=\sum dz_i \wedge \phi_{z_i}$. Let $\tilde{\mcv}=\{(z,w) \in \tilde{\mcw} \vert z \text{ has multiplicity exactly } 2\}$ and note that $\tilde{\mcv}$ is open and dense in $\tilde{\mcw}$. Call $r \colon \tilde{\mcw} \rightarrow S$ the canonical projection onto the space of leaves: we have $cl(r(\tilde{\mcv}))=cl(r(cl(\tilde{\mcv}))=S$, where $cl$ stands for the topological closure. Hence $r(\tilde{\mcv})$ is dense in $S$. Moreover, as the canonical projection onto the space of leaves of a foliation is always an open map \cite[pag.47, Theorem 1]{CLN}, $r(\tilde{\mcv})$ is open in $S$. Since on $\mcv$ we have $\b{V}_{z_i}=\del/\del z_i$ for $i=1, \dots, d$, then $\iota_{\b{V}_z}\Omega=\phi_z$ for every $z$ and it is clear that $\bar{\tau}$ agrees with $\omega$ on $r(\tilde{\mcv})$, hence $\omega$ and $\bar{\tau}$ coincide as claimed on the whole $S$.
\end{proof}

\bibliographystyle{amsplain}
\bibliography{Bibliography_HS}

\end{document}